\documentclass[10pt, leqno]{amsart}

\usepackage{amsfonts,amssymb}
\allowdisplaybreaks[4]

\usepackage{amsmath}
\usepackage{amsthm}
\usepackage{amsrefs}
\usepackage{qsymbols}
\usepackage{latexsym}
\usepackage{chngcntr}
\usepackage[noadjust]{cite}
\usepackage{paralist}
\usepackage{fixltx2e}
\usepackage{esint}

\newtheorem{theorem}{Theorem}[section]

\newtheorem{lemma}[theorem]{Lemma}

\newtheorem{corollary}[theorem]{Corollary}
\theoremstyle{definition}
\newtheorem{definition}[theorem]{Definition}
\newtheorem{remark}[theorem]{Remark}

\newtheorem{innercustomgeneric}{\customgenericname}
\providecommand{\customgenericname}{}
\newcommand{\newcustomtheorem}[2]{%
  \newenvironment{#1}[1]
  {%
   \renewcommand\customgenericname{#2}%
   \renewcommand\theinnercustomgeneric{##1}%
   \innercustomgeneric
  }
  {\endinnercustomgeneric}
}

\newcustomtheorem{assumption}{Assumption}

\newcommand{\IR}{\mathbb{R}}
\newcommand{\IC}{\mathbb{C}}
\newcommand{\IN}{\mathbb{N}}

\newcommand{\IF}{\mathbb{F}}
\newcommand{\IE}{\mathbb{E}}

\newcommand{\cR}{\mathcal{R}}
\newcommand{\cA}{\mathcal{A}}
\newcommand{\cB}{\mathcal{B}}

\newcommand{\HT}{\mathcal{HT}}

\newcommand{\g}{\Gamma}

\newcommand{\tildu}{\widetilde{u}}
\newcommand{\tildv}{\widetilde{v}}

\newcommand{\tildf}{\widetilde{f}}
\newcommand{\tildg}{\widetilde{g}}
\newcommand{\overw}{\overline{w}}

\newcommand{\eps}{\varepsilon}

\numberwithin{equation}{section}

\title[Higher regularity for parabolic equations]{Higher regularity for parabolic equations based on maximal $L_p$-$L_q$ spaces}
\author{Naoto Kajiwara}

\address{Department of Mathematics, Faculty of Science and Technology, Tokyo University of Science, Noda, Chiba, 278-8510, Japan}
\email{kajiwara\_naoto@ma.noda.tus.ac.jp}
\thanks{This work was supported by JSPS KAKENHI Grant Number 19K23408.}

\begin{document}

\begin{abstract}
In this paper we prove higher regularity for $2m$-th order parabolic equations with general boundary conditions. 
This is a kind of maximal $L_p$-$L_q$ regularity with differentiability, i.e. the main theorem is isomorphism between the solution space and the data space using Besov and Tribel--Lizorkin spaces. 
The key is compatibility conditions for the initial data. 
We are able to get a unique smooth solution if the data satisfying compatibility conditions are smooth.  
\end{abstract}

%
\maketitle
\section{Introduction}\label{Intro}

We consider the following parabolic evolution equations; 
\begin{align}\label{eq_para}
\left\{
	\begin{aligned}
		\partial_t u + \omega u + \cA(x, D) u
			&= f(t,x) \quad
				&(t\in \IR_+,~x\in G), \\
		\cB_j(x,D) u
			&= g_j(t,x) \quad 
				&(t\in  \IR_+,~x\in \g, j=1,\ldots, m), \\
		u(0,x)
			&= u_0(x) \quad 
				&(x\in G). 
	\end{aligned}
\right. 
\end{align}
The operators are $\cA(x,D)=\sum_{|\alpha|\le2m} a_\alpha(x)D^\alpha$, $\cB(x,D)=\sum_{|\beta|\le m_j} b_{j\beta}(x) D^\beta$ with $0\le m_j < 2m$, $m \in \IN$, where $D=-i \partial_x$. 
The domain $G$ is bounded or exterior domain with a smooth boundary $\g:=\partial G$. 
Given $f$, $\{g_j\}_{j=1}^m$ and $u_0$, we show the unique solvability of $u$ in the point of maximal $L_p$-$L_q$ regularity, which means the isomorphism between them. 
One character of our paper is to treat higher regularity of data and solution. 
Our main theorem is that the solution belongs to a suitable regularity class when the data are $L_p$ in time and $L_q$ in space with differentiability and vice versa. 
The proof is based on the induction argument. 
To use this argument, we need to focus on the compatibility conditions of the initial data. 
When we require the regularity of the solution we should take care of the initial boundary data. 
This compatibility assumption is essential because we construct the theorem by sufficient and necessary. 

There are a lot of results on maximal regularity and relevant operator theory. 
We refer to \cite{DenkHieberPruss2003, DenkHieberPruss2007, KunstmannWeis2004} for parabolic and elliptic type, \cite{DenkKaip2013} for mixed order type, \cite{Denk-Pruss-Zacher2008} for relaxation type and \cite{FurukawaKajiwara2020} for quasi-steady type. 
In particular we refer to the comprehensive book \cite{PrussSimonett2016}. 
This book is written various method, results and even higher order maximal $L_p$-$L_q$ regularity for parabolic and elliptic equations. 
However only the space regularity for the higher regularity is considered in parabolic case, i.e. base space is $L_p$ in time and $H^s_q$ in space. 
Therefore the compatibility conditions became very complicated. 
In fact they stated the case $s=1$ but they avoided another case $s$. 
On the other hand there is a higher regularity based on maximal $L_2$-$L_2$ regularity in the book \cite{Evans2010} while the equation is second order and zero Dirichlet boundary conditions. 
This is written the result and compatibility clearly. 
The results of our paper is generalizations to the $L_p$-$L_q$ settings and $2m$-th order equation with general boundary conditions. 
The key is to set the base space $H^k_p(\IR_+; L_q(G))\cap L_p(\IR_+; H^{2mk}_q(G))$. 
The function spaces of initial data and boundary data are used Besov and Triebel--Lizorkin spaces. 
This is similar to the classical case $k=0$, but we give the proof again. 
The compatibility conditions are required for the differentiability index $k$. 
 By considering $k$ inductively, we are able to get the smooth solution as a corollary of the main theorem. 
 
 This paper is organized as follows. 
 In section \ref{Main} we consider the equation more precisely. 
 After setting the situation and assumption, we state our main theorem. 
 In section \ref{Pre} we prepare for the definitions and lemmas to prove the necessity of the theorem. 
 In section \ref{proof} we prove the main theorem. 
 The sufficient conditions are proved by the induction argument as we can see the book \cite{Evans2010}. 
 The necessity conditions are almost similar to the results on \cite{DenkHieberPruss2007, PrussSimonett2016} which is explained in the section  \ref{Pre}. 

\section{Settings and main theorem}\label{Main}
\subsection{Some settings}
At first we define some function spaces to state our main theorems. 
Few terminologies are given in the section \ref{Pre}. 
For almost all definitions and notations, see \cite{DenkKaip2013, PrussSimonett2016, Lunardi2009, Triebel1983}. 

Let $X$ be a Banach space, $\IN_0:=\IN\cup\{0\}$ and $r\geq 0, 1<p,q<\infty$. 
We define the {\it vector-valued Besov and Triebel--Lizorkin spaces} by 
\begin{align*}
B^r_{p,q}(\IR^n; X)
	&:= \{u\in\mathcal{S}'(\IR^n; X)\mid \|u\|_{B^r_{p,q}(\IR^n; X)}:=\|(2^{rj} \|\mathcal{F}^{-1}[\phi_j \mathcal{F} u]\|_{L_p(\IR^n; X)})_{j\in\IN_0}\|_{l_q}<\infty\}, \\
F^r_{p,q}(\IR^n; X)
	&:= \{u\in\mathcal{S}'(\IR^n; X)\mid \|u\|_{F^r_{p,q}(\IR^n; X)}:=\|\|(2^{rj} \mathcal{F}^{-1}[\phi_j \mathcal{F} u])_{j\in\IN_0}\|_{l_q(X)}\|_{L_p(\IR^n)}<\infty\}, 
\end{align*}
where $\{\phi_j\}_{j\in\IN_0}$ is Littlewood--Paley smooth dyadic decomposition. 
Moreover for a domain $\Omega\subset \IR^n$ we define $B^r_{p,q}(\Omega; X)$ and $F^r_{p,q}(\Omega; X)$ by a restriction of $B^r_{p,q}(\IR^n; X)$ and $F^r_{p,q}(\IR^n; X)$ to $\Omega$. 

Note that 
\begin{align*}
B^r_{p,q}(\Omega; X) &= (H^{k_0}_p(\Omega; X), H^{k_1}_p(\Omega; X))_{\theta, q},\quad k_0, k_1\in\IN_0,~\theta\in(0,1),~r:=(1-\theta)k_0 + \theta k_1, 
\end{align*}
where $(\cdot, \cdot)_{\theta, q}$ is real interpolation. 
We use this relation later. 

\begin{definition}
(a) A Banach space $X$ is said to be of class $\HT$ if the Hilbert transform $H$ defined by 
\begin{align*}
Hf(t) :=\frac{1}{\pi} \lim_{R\to\infty} \int_{R^{-1}\le |s| \le R} f(t-s) \frac{ds}{s}
\end{align*}
is bounded on $L_p(\IR; X)$ for some $p\in(1, \infty)$. \\
(b) A Banach space $X$ is said to have property $(\alpha)$ if there exists a constants $C>0$ such that 
\begin{align*}
\left|\sum_{i,j=1}^N \alpha_{ij}\eps_i \eps_j^{\prime} x_{ij}\right|_{L_2(\Omega\times\Omega'; X)}
\le C \left|\sum_{i,j=1}^N \eps_i \eps_j^{\prime} x_{ij}\right|_{L_2(\Omega\times\Omega'; X)}
\end{align*}
for all $\alpha_{ij}\in\{-1, 1\}$, $\{x_{ij}\}_{i,j=1}^N\subset X$, positive integer $N$, and all symmetric independent $\{-1, 1\}$-valued random variables $\eps_i$ (respectively $\eps_j^\prime$) on a probability space $(\Omega, \cA, \mu)$ (respectively $(\Omega', \cA', \mu')$). \\
(c) $\HT(\alpha)$ denotes the class of Banach spaces which belong to $\HT$ and have property $(\alpha)$. 
\end{definition}
We consider the function $u(t,x)\in E\in \HT(\alpha)$ and note that $\IC^{N\times N}$ is a Banach space of class $\HT(\alpha)$, so we are able to consider a system of usual parabolic equations. 

We prepare for function spaces of $f, \{g_j\}_{j=1}^m$. 
Let $\kappa_j = 1 - \frac{m_j}{2m} - \frac{1}{2mq}$, 
\[
\IF^k_1 := H^k_p(\IR_+; L_q(G))\cap L_p(\IR_+; H^{2mk}_{q}(G)), \qquad
\IF^{k, m_j}_2:= F^{k+\kappa_j}_{p,q}(\IR_+; L_q(\g)) \cap L_p(\IR_+; B^{2mk+\kappa_j}_{q,q}(\g)). 
\]

We need to consider the operators $(\cA(x, D), \cB_1(x,D), \ldots, \cB_m(x,D))$. 
For the regularity of the operators, we assume the following assumptions ${\rm (R_k)}$; \\
${\rm (R_1)}$ If $k=0$, for each $|\alpha|=2m$ $a_\alpha$ is a bounded continuous, and ${\displaystyle\lim_{|x|\to\infty}} a_{\alpha}$ exists if $\Omega$ is unbounded. \\
${\rm (R_2)}$ $a_\alpha\in H^k_{r_l}(\Omega; \cB(E)) + H^k_\infty(\Omega; \cB(E))$ for each $|\alpha|=l\le 2m$, with $r_l \ge q$ and $2m + k -l >n/r_l$. \\
${\rm (R_3)}$ $b_{j\beta} \in B^{2m\kappa_j + k}_{r_{jl}, q}(\g; \cB(E))$ for a each $|\beta|=l<m_j$, with $r_{jl}\geq q$, and $2m\kappa_j + k > (n-1)/r_{jl}$. 
Above regularity assumption is same as higher order elliptic problems. 
See \cite{PrussSimonett2016}. 

\begin{definition}
We call the system $(\cA(x, D), \cB_1(x,D), \ldots, \cB_m(x,D))$ uniformly normally elliptic if \\
(i) $\cA(x,D)$ is normally elliptic in the sense that infimum of an angle $\phi\in[0, \pi)$ such that $\sigma(\cA(x,\xi))\subset \Sigma_\phi$ for all $\xi\in\IR^n, |\xi|=1$ is less than $\pi/2$, for each $x\in \overline{G}\cup \{\infty\}$. \\
(ii) The following the Lopatinskii--Shapiro condition (LS) holds, for each $x\in \g$. \\
(LS) We rewrite the equations (\ref{eq_para}) in coordinates associated with $x$ so that the positive part of $x_n$-axis has the direction of the inner normal at $x$ after a transformation and a rotation.
For all $(\lambda, \xi')\in (\Sigma_\theta \times\IR^{n-1} )\setminus \{ (0,0) \}$ ($\theta > \pi / 2$) and $\{g_j\}_{j=0}^m\in E^m$, the ODEs on the half line $\IR_+$ given by 
\begin{align}\label{eq_ls}
\left \{\begin{array}{rll}
	\lambda v(y) + \cA_\#(x,\xi', D_y) v(y) 
				&= 0 \quad &(y>0), \\
	\cB_{j\#}(x,\xi', D_y) v(0)
				&= g_j \quad 
				&(j = 1, \cdots, m)
		\end{array}
		\right.
\end{align}
admit a unique solution $v \in C_0^{2m}( \IR_+; E)$, where the subscript $\#$ denotes the principal part of the corresponding operator $\cA_\#(x,D) =\sum_{|\alpha|=2m} a_\alpha(x) D^\alpha, \cB_{j\#}(x,D) =\sum_{|\beta|=m_j} b_{j\beta} (x) D^\beta$ and 
\begin{align*}
\Sigma_{\phi}&:=\{z\in\IC\setminus\{0\}\mid |\arg z|<\phi\}, \\
C^{2m}_0 ( \IR_+; E) &:= \left \{ v \in C^{2m} ( \IR_+; E) \, ; \, \lim_{y \rightarrow \infty} v (y) = 0 \right \}. 
\end{align*}
\end{definition}

\subsection{main theorem}
\begin{theorem}\label{main}
Let $G\subset \IR^n$ be open with compact boundary of class $C^{2m+k}$, $1<p,q<\infty$, and let $E$ be a Banach space of class $\mathcal{HT}(\alpha)$. 
Assume that $(\cA(x, D), \cB_1(x,D), \ldots, \cB_m(x,D))$ is uniformly normally elliptic and satisfies $({\rm R_k})$. 
Let $\kappa_j \neq 1/p$ for all $j$. 
Then there is $\omega_0\in\IR$ such that for each $\omega>\omega_0$, equation \eqref{eq_para} admits a unique solution $u$ in the class 
\[ u \in \IE^k := H^{k+1}_p(\IR_+; L_q(G))\cap L_p(\IR_+; H^{2m(k+1)}_{q}(G)), \]
if and only if the data are subject to the following conditions. 
\begin{align*}
&(f, \{g_j\}_{j=1}^m) 
	\in \IF^k_1 \times \prod_{j=1}^m \IF^{k, m_j}_2, \\
&v_0:=u_0
	\in B^{2m(k+1-\frac{1}{p})}_{q,p}(G)
		\quad {\rm with}\quad \cB_j v_0 = g_j|_{t=0}, \\
&v_1:=f|_{t=0}- (\omega+\cA(x,D)) v_0
 	\in B^{2m(k-\frac{1}{p})}_{q,p}(G)
 		\quad {\rm with}\quad \cB_j v_1 = (\partial_t g_j)|_{t=0}, \\
&v_2:=(\partial_t f)|_{t=0}- (\omega + \cA(x,D)) v_1
 	\in B^{2m(k-1-\frac{1}{p})}_{q,p}(G)
 		\quad {\rm with}\quad \cB_j v_2 = (\partial_t^2 g_j)|_{t=0}, \\
&\qquad \vdots\\
&v_{k-1}:=(\partial_t^{k-2} f)|_{t=0} - (\omega + \cA(x,D))v_{k-2} 
	\in B^{2m(2-\frac{1}{p})}_{q,p}(G)
		\quad {\rm with}\quad \cB_j v_{k-1} = (\partial_t^{k-1} g_j)|_{t=0}, \\
&\qquad \rm{and}\\
&v_k:=(\partial_t^{k-1}f)|_{t=0} - (\omega + \cA(x,D))v_{k-1} 
	\in B^{2m(1-\frac{1}{p})}_{q,p}(G) 
		\quad {\rm with}\quad \cB_j v_k = (\partial_t^k g_j)|_{t=0}
	\quad {\rm if} \quad \kappa_j >\frac{1}{p}. 
\end{align*}
The solution depends continuously on the data in the corresponding spaces, i.e. there is $C>0$ such that 
\[\|u\|_{\IE^k} \leq C \left(\|f\|_{\IF^k_1} + \sum_{j=1}^m \|g_j\|_{\IF^{k,m_j}_2} + \|u_0\|_{B^{2m(k+1-\frac{1}{p})}_{q,p}(G)} \right) \]
for all $(f, g)$ satisfying conditions. 
\end{theorem}

\begin{remark}
Since we would like to construct the global estimate, i.e. time interval is $\IR_+$, we need the term $\omega u$ in the equation \eqref{eq_para}. 
If we replace the time interval $[0, T]$ for some $T>0$, we are able to set $\omega =0$. 
\end{remark}

\begin{corollary}
Let $T>0$. 
Assume $\Omega$ is a open with compact boundary of class $C^\infty$, $f\in C^\infty([0, T] \times \overline{G})$, $g_j \in C^\infty([0, T]\times \g)$ for all $j$, $u_0\in C^\infty(\overline{G})$ and the $k$-th order compatibility conditions hold for $k=1, 2, \cdots$. 
Then the parabolic equations \eqref{eq_para} has a unique solution 
\[ u\in C^\infty([0, T] \times \overline{G}). \]
\end{corollary}

\begin{proof}
Apply theorem \ref{main} for $k=0,1,\ldots$ and use embedding theorem. 
\end{proof}

\section{Preliminaries}\label{Pre}

\subsection{Trace to the initial data}
In this subsection we consider the trace to the initial data. 
This needs the necessity of the maximal regularity. 
By the real interpolation theory, we are able to characterize the initial space as well as the usual $k=0$ case. 
This is a famous result, but we reconsider the higher order regularity. 
The method is almost same. 

\begin{definition}
Let $A$ be a closed linear operator in $X$. 
$A$ is called sectorial if the following two conditions are satisfied. \\
${\rm (i)}$~$\overline{D(A)}=\overline{R(A)}=X$, $(-\infty, 0)\subset \rho(A)$. \\
${\rm (ii)}$~$|t(t+A)^{-1}| \le M$ for all $t>0$, and some $M>0$. \\
If $(-\infty, 0)\subset \rho(A)$ and only (ii) hold then $A$ is said to be pseudo-sectorial. 
For the pseudo-sectorial operator, spectral angle is given by 
\[\phi_A:=\inf\{\phi \mid \rho(-A)\supset\Sigma_{\pi-\phi},~\sup_{\lambda\in\Sigma_{\pi-\phi}}|\lambda(\lambda+A)^{-1}|<\infty\}. \]
\end{definition}

\begin{definition}
Let $A$ be a densely defined pseudo-sectorial operator with spectral angle $\phi_A<\pi/2$ in $X$. 
Let $\alpha\in(0,1)$ and $p\in[1, \infty)$. 
The space $D_A(\alpha, p)$ are defined by means of 
\[D_A(\alpha, p) := \{x\in X \mid [x]_{\alpha,p}:= (\int_0^\infty |t^{1-\alpha} Ae^{-tA}x|^p\frac{dt}{t})^{1/p}<\infty\}.\]
When equipped with the norm 
\[|x|_{\alpha, p} := |x| + [x]_{\alpha,p},\quad x\in D_A(\alpha, p),\]
$D_A(\alpha, p)$ becomes a Banach space. 
For $k\in \IN$ the spaces $D_A(k+\alpha, p)$ are defined by 
\[D_A(k+\alpha, p):=\{x\in D(A^k)\mid A^k x\in D_A(\alpha,p)\}, \quad |x|_{k+\alpha,p}:=|x|+[A^kx]_{\alpha,p}.\]
\end{definition}

\begin{lemma}\label{initial trace}
Suppose $A$ is a densely defined invertible sectorial operator in $X$ with spectral angle $\phi_A<\pi/2$, $p\in(1,\infty)$ and $k\in\IN\cup\{0\}$. 
Then for the function $u(t):=e^{-tA}x$ the following assertions are equivalent. \\
$(a)$~$u(t)\in D(A^{k+1})$ for a.e. $t>0$, and $u\in L_p(\IR_+; D(A^{k+1}))$. \\
$(b)$~$u\in H^{k+1}_p(\IR_+; X)$\\
$(c)$~$x\in D_A(k+1-1/p, p)$\\
In this case there is a constant $C$ depending only on $A$ and $p$ such that 
\[|u|_{H^{k+1}_p(\IR_+; X)} + |u|_{L_p(\IR_+; D(A^k))} \le C |x|_{k+1-1/p,p}.\]
\end{lemma}

\begin{proof}
By a standard semigroup theory, we have $e^{-tA}X\subset D(A^{k+1})$ for $t>0$ and, with some $\omega>0$, 
\[|e^{-tA}| + t|Ae^{-tA}| \le Me^{-\omega t}, \quad t>0.\]
By definition, $x\in D_A(k+1-1/p,p)$ implies $A^k x \in D_A(1-1/p,p)$, i.e. $Ae^{-tA}(A^k x) \in L_p(\IR_+; X)$. 
This and commutativity of $A$ and $e^{-tA}$ mean $e^{-tA}x \in L_p(\IR_+; D(A^{k+1}))$, hence $(c)\Rightarrow (a)$. 
Since $e^{-tA}$ is holomorphic and $\frac{d}{dt} e^{-tA}=-Ae^{-tA}$ for $t>0$, $\frac{d^{k+1}}{dt^{k+1}}u = (-A)^{k+1} u \in L_p(\IR_+; X)$ if we assume $(a)$, hence $(a)\Rightarrow (b)$. 
On the other hand, $(b)$ yields $A^{k+1}u=(-1)^{k+1} \frac{d^{k+1}}{dt^{k+1}} u \in L_p(\IR_+;X)$ and 
\[[A^kx]_{1-1/p,p}^p = |A^{k+1}u|_{L_p(\IR_+;X)}^p. \]
This is $(b)\Rightarrow (c)$. 
\end{proof}

\subsection{Trace to the boundary data}
In this subsection we consider the trace to the boundary data. 
This also needs the necessity of the maximal regularity. 
We cite the results on Denk--Hieber--Pr\"uss \cite{DenkHieberPruss2007}, which characterize the boundary trace as Triebel--Lizorkin space. 
See also \cite{PrussSimonett2016}. 

Let $L_0:=\omega + \partial_t + (-\Delta_x)^m$ in the space $X_0:= L_p(\IR_+; L_q(\IR^{n-1};E))$ with domain 
\[D(L_0):= {}_0 H^1_p(\IR_+; L_q(\IR^{n-1}; E)) \cap L_p(\IR_+; H^{2m}_q(\IR^{n-1}; E)). \]
This operator is a sectorial operator with angle $\pi/2$ and $L_0^{1/2m}$ is the negative generator of an analytic $C_0$-semigroup $e^{-y L_0^{1/2m}}$. 
Let $L$ be the canonical extension of $L_0$ to the space $L_p(\IR_+, L_q(\IR^n_+;E))$. 
\begin{lemma}\label{boundary trace}
Let $1<p,q < \infty$ and $E$ be a Banach space with property $\mathcal{HT}(\alpha)$. 
Moreover, let $L_0$ and $L$ be defined as above, and let $u(y) := e^{-yL_0^{1/2m}}g$, $g\in X_0$, $y>0$. 
Then the following assertions are equivalent. \\
${\rm (i)}$~$u\in{}_0 H^{1/2m}_p(\IR_+; L_q(\IR_+\times \IR^{n-1}; E))\cap L_p(\IR_+; H^1_q(\IR_+\times \IR^{n-1}; E))$\\
${\rm (ii)}$~$L^{1/2m}u\in L_p(\IR_+; L_q(\IR_+\times \IR^{n-1}; E))$\\
${\rm (iii)}$~$g\in {}_0 F^{1/2m-1/2mq}_{p,q}(\IR_+; L_q(\IR^{n-1}; E))\cap L_p(\IR_+; B^{1-1/q}_{q,q}(\IR^{n-1};E))$. \\
Similar statements are valid on $\IR$; replace the symbols $L_p(\IR_+;\cdot)$ and ${}_0K_p(\IR_+; \cdot)$ by $L_p(\IR; \cdot)$ and $K_p(\IR; \cdot)$, respectively. 
\end{lemma}

\subsection{Higher regularity for the elliptic equations}
This subsection we collect the results on higher regularity for the elliptic equations. 
All results are written in \cite{PrussSimonett2016}. 
We do not need the compatibility conditions, which is different from parabolic problems. 

\begin{theorem}\label{elliptic}
If $f\in H^k_q(\Omega; E)$ and $g_j\in B^{2m\kappa_j + k}_{q,q}(\g; E)$, then the solution of the elliptic problem 
\begin{align}\label{eq_ellip}
\left\{
	\begin{aligned}
		(\omega + \cA(x, D)) u
			&= f(x) \quad
				&(x\in G), \\
		\cB_j(x,D) u
			&= g_j(x) \quad 
				&(x\in \g, j=1,\ldots, m)
	\end{aligned}
\right. 
\end{align}
has a unique solution in $H^{2m+k}_q(\Omega; E)$, provided $\cA(x,D)$ is normally elliptic, the Lopatinskii--Shapiro condition ${\rm (LS)}$ holds, $\omega > {\sf s}(-A):=\sup{\mathrm {Re}}\,\sigma(-A)$, $\g \in C^{2m+k}$, and the coefficients satisfy the regularity conditions ${\rm (R_k)}$. 
\end{theorem}

\section{Proof of the theorem \ref{main}} \label{proof}
\begin{proof}[Proof of a sufficiency of the main theorem]
The proof is based on an induction on $k$. 
For the sake of simplicity of notation we prove the theorem up to the case $k+1$. 
The case $k=0$ is just a result of \cite{DenkHieberPruss2007}. 
Assume a sufficient condition of the theorem is valid for some nonnegative integer $k$, and suppose the data $(f, \{g_j\}_{j=1}^m, u_0)$ satisfy the $k+1$-th compatibility conditions. 
By differentiate the equation with respect $t$, we consider, by setting $\tildu:=\partial_t u, \tildf:=\partial_t f$ and $\tildg_j:=\partial_t g_j$, 
\begin{align}
\left\{
	\begin{aligned}
		\partial_t \tildu + \omega \tildu+\cA(x, D) \tildu
			& = \tildf(t,x) \quad
				&(t\in J,~x\in G), \\
		\cB_j(x,D) \tildu
			&= \tildg_j(t,x) \quad 
				&(t\in  J,~x\in \g, j=1,\ldots, m), \\
		\tildu(0,x)
			&= f(0,x) - (\omega + \cA(x,D))u_0 =: \tildu_0 \quad 
				&(x\in G). 
	\end{aligned}
\right. 
\end{align}
Note that 
\begin{align*}
&(\tildf, \{\tildg_j\})\in \IF^k_1 \times \prod_{j=1}^m \IF^{k, m_j}_2, \\
&\tildv_0:=\tildu_0
	\in B^{2m(k+1-\frac{1}{p})}_{q,p}(G)
		\quad {\rm with}\quad \cB_j \tildv_0 = \tildg_j|_{t=0}, \\
&\tildv_1:=\tildf|_{t=0}- (\omega + \cA(x,D)) \tildv_0
 	\in B^{2m(k-\frac{1}{p})}_{q,p}(G)
 		\quad {\rm with}\quad \cB_j \tildv_1 = (\partial_t \tildg_j)|_{t=0}, \\
&\tildv_2:=(\partial_t \tildf)|_{t=0}- (\omega + \cA(x,D)) \tildv_1
 	\in B^{2m(k-1-\frac{1}{p})}_{q,p}(G)
 		\quad {\rm with}\quad \cB_j \tildv_2 = (\partial_t^2 \tildg_j)|_{t=0}, \\
&\qquad \vdots\\
&\tildv_{k-1}:=(\partial_t^{k-2}\tildf)|_{t=0} - (\omega  + \cA(x,D))\tildv_{k-2} 
	\in B^{2m(2-\frac{1}{p})}_{q,p}(G)
		\quad {\rm with}\quad \cB_j \tildv_{k-1} = (\partial_t^{k-1} \tildg_j)|_{t=0}, \\
&\qquad \rm{and}\\
&\tildv_k:=(\partial_t^{k-1}\tildf)|_{t=0} - (\omega + \cA(x,D))\tildv_{k-1} 
	\in B^{2m(1-\frac{1}{p})}_{q,p}(G)
		\quad  {\rm with}\quad \cB_j \tildv_k = (\partial_t^k \tildg_j)|_{t=0}
	\quad {\rm if} \quad \kappa_j >\frac{1}{p}. 
\end{align*}
Namely $(\tildf, \{\tildg_j\}_{j=1}^m, \tildu_0)$ satisfy the $k$-th compatibility. 
Thus applying the induction assumption, we deduce $\tildu \in \IE^{k}$ and 
\[\|\tildu\|_{\IE^k} \leq C \left(\|\tildf\|_{\IF^k_1} + \sum_{j=1}^m \|\tildg_j\|_{\IF^{k,m_j}_2} + \|\tildu_0\|_{B^{2m(k+1-\frac{1}{p})}_{q,p}(G)} \right) . \]
This implies 
\[\partial_t u\in \IE^k = H^{k+1}_p(J; L_q(G))\cap L_p(J; H^{2m(k+1)}_{q}(G))\]
and 
\begin{align}
\|\partial_t u\|_{\IE^k} 
&\leq C \left(\|\partial_t f\|_{\IF^k_1} + \sum_{j=1}^m \|\partial_t g_j\|_{\IF^{k,m_j}_2} + \|f(0) - (\omega + \cA(x,D))u_0\|_{B^{2m(k+1-\frac{1}{p})}_{q,p}(G)} \right) \nonumber\\
&\leq C\left(\|f\|_{\IF^{k+1}_1} + \sum_{j=1}^m \|g_j\|_{\IF^{k+1, m_j}_2} + \|u_0\|_{B^{2m(k+2-\frac{1}{p})}_{q,p}(G)} \right) \label{time es}. 
\end{align}
Here we used the estimate 
\[\|f|_{t=0}\|_{B^{2m(k+1-\frac{1}{p})}_{q,p}(G)} \leq C\|f\|_{\IF^{k+1}_1}. \]

We next consider the following the higher order {\it elliptic problems} for each fixed $t$; 
\begin{align}
\left\{
	\begin{aligned}
		(\omega + \cA(x, D)) u
			&= f(t,x) - \partial_t u \quad
				&(x\in G), \\
		\cB_j(x,D) u
			 &= g_j(t,x) \quad 
				&(x\in \g, j=1,\ldots, m). 
	\end{aligned}
\right. 
\end{align}
Under the assumptions on the regularity of the coefficients ${\rm (R_{k+1})}$, we have $u(t) \in H^{2m(k+2)}_q(G)$ by theorem \ref{elliptic} and 
\begin{align*}
\|u(t)\|_{H^{2m(k+2)}_{q}(G)} 
&\leq C \left(\|f(t) - \partial_t u(t) \|_{H^{2m(k+1)}_q(G)} + \sum_{j=1}^m \|g_j(t)\|_{B^{2m(k+2)-m_j-\frac{1}{q}}_{q,q}(\g)} + \|u(t)\|_{L_q(G)} \right) \\
&\leq C\left( \|f(t)\|_{H^{2m(k+1)}_q(G)} + \|\partial_t u(t)\|_{H^{2m(k+1)}_q(G))} + \sum_{j=1}^m \|g_j(t)\|_{B^{2m(k+1)+\kappa_j}_{q,q}(\g)}+ \|u(t)\|_{L_q(G)}\right)
\end{align*}
for a.e. $t$. 
We take $L_p(J)$, then, using the estimate \eqref{time es},  
\begin{align*}
&\|u\|_{L_p(J; H^{2m(k+2)}_{q}(G))} \\
\leq& C \left( \|f\|_{L_p(J; H^{2m(k+1)}_q(G))} + \|\partial_t u\|_{L_p(J; H^{2m(k+1)}_q(G))} + \sum_{j=1}^m \|g_j\|_{Lp(J; B^{2m(k+1)+\kappa_j}_{q,q}(\g))}+ \|u\|_{L_p(J; L_p(G))}\right) \\
\leq& C \left( \|f\|_{\IF^{k+1}_1} + \sum_{j=1}^m \|g_j\|_{\IF^{k+1,m_j}_2} + \|u_0\|_{B^{2m(k+2-\frac{1}{p})}_{q,p}(G)} + \|u\|_{L_p(J; L_q(G))}\right). 
\end{align*}
We again consider the estimate \eqref{time es}, we deduce 
\[\|u\|_{\IE^{k+1}} \leq C \left(\|f\|_{\IF^{k+1}_1} + \sum_{j=1}^m \|g_j\|_{\IF^{k+1,m_j}_2} + \|u_0\|_{B^{2m(k+2-\frac{1}{p})}_{q,p}(G)}\right)\]
since we have the estimate 
\[\|u\|_{L_p(J; L_q(G))} \leq C \left(\|f\|_{\IF^{k+1}_1} + \sum_{j=1}^m \|g_j\|_{\IF^{k+1,m_j}_2} + \|u_0\|_{B^{2m(k+2-\frac{1}{p})}_{q,p}(G)}\right). \]
This proves the sufficient conditions to get the higher order maximal $L_p$-$L_q$ regularity solutions. 
\end{proof}

\begin{proof}[Proof of a neccessity of the main theorem]
For the sake of simplicity we prove the case that $G=\IR^n_+$. 
The desired case is straightforward by coordinate transformation.  
Let $u\in \IE^k$ be the solution of \eqref{eq_para}. 
Then we have $f\in \IF^k_1$. 
We extend the function $u$ in space so that $E_x u \in H^{k+1}_p(J; L_q(\IR^n))\cap L_p(J; H^{2m(k+1)}_{q}(\IR^n))$, where $E_x$ is an extension operator from $\IR^n_+$ to $\IR^n$. 
From the lemma \ref{initial trace}, we know 
\[ E_x u|_{t=0}\in B^{2m(k+1-\frac{1}{p})}_{q,p}(\IR^n; E),\]
which implies $u_0\in B^{2m(k+1-\frac{1}{p})}_{q,p}(\IR^n_+; E)$. 
Moreover, since $\partial_t u = f - (\omega + \cA) u=:v_1$ and 
\begin{align*}
\partial_t^2 u 
	&= \partial_t (f - (\omega +\cA) u)\\
	&= \partial_t f - (\omega+\cA)v_1=:v_2\\
\partial_t^3 u 
	&= \partial_t (\partial_t f - (\omega+\cA)v_1)\\
	&= \partial^2_t f - (\omega + \cA)v_2=:v_3\\
\vdots&
\end{align*}
at $t=0$ derives the compatibility conditions whose regularity is same as the theorem. 

Next we consider the tangential part of $y:=x_n=0$. 
We extend the function $u$ in time so that $E_t u =:v \in H^{k+1}_p(\IR; L_q(\IR^n_+))\cap L_p(\IR; H^{2m(k+1)}_{q}(\IR^n_+))$, where $E_t$ is an extension operator from $\IR_+$ to $\IR$. 
Hence $w:=(\omega + \partial_t)^\alpha \partial_y ^l D_x^\beta v$ belongs $H^{1/2m}_p(\IR; L_q(\IR^n_+;E)\cap L_p(\IR; H^1_q(\IR^n_+;E))$ if $2m\alpha + l + |\beta| = k + 2m - 1$. 
Define the function $\overw$ by the solution of 
\[\partial_y \overw + L_0^{1/2m}\overw = \partial_y w + L_0^{1/2m} w,~y>0,\quad\overw(0)=0.\]
This function also belongs to $H^{1/2m}_p(\IR; L_q(\IR^n_+;E)\cap L_p(\IR; H^1_q(\IR^n_+;E))$ since  $L_0$ satisfies the maximal regularity, see PS. 
Hence $w-\overw = e^{-y L^{1/2m}_0} w|_{y=0}$ has same regularity as well. 
Then we are able to use lemma \ref{boundary trace} to get 
\[w|_{y=0} \in F^{\frac{1}{2m}-\frac{1}{2mq}}_{p,q}(\IR; L_q(\IR^{n-1};E)) \cap L_p(\IR; B^{1-\frac{1}{q}}_{q,q}(\IR^{n-1};E)). 
\]
By definition of $w$ and proper choices of $\beta$ and $l$, this yields 
\[\cB_j(x,D) v \in F^{k+\kappa_j}_{p,q}(\IR; L_q(\IR^{n-1};E)) \cap L_p(\IR; B^{2mk + 2m\kappa_j}_{q,q}(\IR^{n-1};E)), \]
by restriction to $t>0$; we finally obtain $g_j \in \IF^{k,m_j}_2$. 
This complete the necessity of the main theorem. 
\end{proof}

\end{document}